\documentclass[reqno,12pt]{amsart} 

\usepackage{amssymb}
\usepackage{amsmath}

\theoremstyle{plain}
\newtheorem {lemma}{Lemma}[]
\newtheorem {theorem}[lemma]{Theorem}

\newtheorem {proposition}[lemma]{Proposition}

\theoremstyle{remark}
\newtheorem* {remark}{Remark}

\newtheorem {example}[lemma]{Example}

\theoremstyle{definition}

\newcommand{\qt}{(\frac{a,b}{F})}
\newcommand{\oqt}{(\frac{-1,-1}{F})}

\newcommand\SO[1][n]{{\operatorname{SO}_{#1}}}
\newcommand\Orth[1][n]{{\operatorname{O}_{#1}}}
\newcommand\K[1][1]{{\operatorname{K}_{#1}}}
\newcommand{\qi}{\mathbf i}
\newcommand{\qj}{\mathbf j}
\newcommand{\qk}{\mathbf k}
\newcommand{\mQ}{\mathcal Q}
\newcommand{\mH}{\mQ}
\newcommand{\rr}{\mathbb R}

\newcommand\CK[1][1]{\operatorname{CK}_{#1}}
\newcommand\SK[1][1]{\operatorname{SK}_{#1}} 
\newcommand\Br{\operatorname{Br}}

\def\F{\mathbb F}
\def\Q{\mathbb Q}
\def\R{\mathbb R}
\def\Z{\mathbb Z}
\DeclareMathOperator{\Gal}{Gal} %

\newcommand{\Nrd}[1][{}]{{\operatorname{Nrd}_{#1}}} 
\newcommand{\Trd}[1][{}]{{\operatorname{Trd}_{#1}}} 

\newcommand\ol{\overline}
\newcommand\mychar{{\operatorname{char}}}

\newcommand\GL[1][d]{{\operatorname{GL}_{#1}}} 

\newcommand\co{{\,:\,}}

\def\divides{{\,|\,}}           

\newcommand\DEG[2]{[{#1}\,{:}\,{#2}]}
\newcommand\paper[1]{{\it {#1}}}

\long\def\forget#1\forgotten{}

\begin{document}

\title[Maximal subgroups of division algebras]{On maximal Subgroups of the multiplicative group of a division algebra }

\author{R. Hazrat}
\address{
Department of Pure Mathematics\\
Queen's University\\
Belfast BT7 1NN\\
United Kingdom} \email{r.hazrat@qub.ac.uk}

\author{A. R. Wadsworth}
\address{
Department of Mathematics\\
University of California at San Diego\\
La Jolla, California 92093-0112\\
U.S.A.}
 \email{arwadsworth@ucsd.edu}


\begin{abstract} The question of existence of a maximal subgroup in the
multiplicative group~$D^*$ of a division algebra $D$ finite dimensional
over its center $F$ is investigated. We prove that if $D^*$~has no
maximal subgroup, then $\deg(D)$ is not a power of $2$, $F^{*2}$
is divisible, and for each odd prime~$p$
dividing $\deg(D)$, there exist noncyclic
division algebras of degree $p$ over~$F$. 
\end{abstract}

\maketitle

\section{Introduction}

Let $D$ be a non-commutative division ring with center $F$. The
structure of subgroups of the multiplicative group
$D^*=D\backslash \{0\}$, in general, is unknown. Finite subgroups
of $D^*$ have been classified by Amistur \cite{amt}. Normal and
subnormal subgroups of $D^*$ have been studied over the last 70
years. Herstein (\cite{lam}, 13.26) showed that the number of
conjugates of a non-central element of $D$ is infinite. (In fact
it has the same cardinal number as $D$, \cite{scott}). This
implies that a non-central normal subgroup of a division ring is
``big.'' Confirming this, Stuth \cite{stuth} proved that if an
element commutes with a non-central subnormal subgroup of a
division ring, then it is central. In fact he proved that if
$[x,G] \subseteq F$ where $G$ is a subnormal subgroup of $D^*$ and
$ [x,G]=\{xgx^{-1}g^{-1} \mid g \in G\}$ then $x \in F$. He
concluded that a subnormal subgroup of a division ring could not
be solvable. Another remarkable result
has recently been obtained in
 major work  by Rapinchuk, Segev and
Seitz  \cite{rapsegec}. They showed that a normal subgroup of
finite dimensional division ring which has a finite quotient in
$D^*$ contains one of the groups appearing in the derived series of
$D^*$, i.e., the quotient group itself is solvable.

Now, as with the normal subgroups, one would like to know the
structure of maximal subgroups of $D^*$ and how ``big'' they are
in $D^*$. A maximal subgroup of a nilpotent group is normal.
However $D^*$ is not solvable and thus not nilpotent. Indeed, there
exist division algebras which contain non-normal maximal subgroups
(see Section \ref{ck1} below). The recent papers
\cite{akbari4,akbarimah,akbari,ebrahimi,kesh,mahdavi} study
various aspects of maximal subgroups in the multiplicative group
of a division ring. But, the question of  existence of maximal subgroups
in an arbitrary division ring has not been settled.
The most extensive previous result in this direction
was proved by Keshavarzipour and Mahdavi-Hezavehi.
They showed  in Cor.~2
of  \cite{kesh} that if  $D$ is a division algebra with center
$F$, and with  prime power degree $p^n$, and $D$~is not a quaternion
algebra, then $D^*$
has a maximal subgroup
if  $\mychar(F) = 0$  or $\mychar(F) = p$ or
$F$ contains a primitive $p$-th root of unity.

In this note we investigate the question of existence of a maximal
subgroup in the  multiplicative group of a division algebra finite dimensional
over its center. The general approach is to consider the
$K$-functor $\CK(A) = \text{coker}(\K(F) \rightarrow \K(A) )$ for
the central simple algebra $A=M_t(D)$ with center $F$.
Whenever $F$ is infinite, we have $\CK(A) \cong D^*\big/ F^{*t}D'$, where
$D'$~denotes the derived group of $D^*$.
The group $\CK(A)$ is
abelian  of bounded exponent and when it is nontrivial it
gives rise to (normal) maximal subgroups in $D^*$ (see Section
\ref{ck1}).  For
quaternion algebras $\mQ$ over euclidean fields, separate
treatment is required, since
we'll see that $\CK(M_t(\mQ))$
can be trivial for all $t$.

This  paper is organized  as follows: In Section
\ref{ck1} we examine the relation between the functor $\CK$ and
the maximal subgroups of multiplicative group of a division
algebra. We prove that for a quaternion division algebra $\mQ$,
$\CK(M_2(\mQ))$ is trivial if and only if ${\mQ = \oqt}$ where
$F$ is a euclidean field, if and only if
$\mQ^*$ has no normal maximal subgroup of index $2$.
Using valuation theory, we also
provide examples of non-normal maximal subgroups of finite index
in division algebras. Indeed, for any prime power $q$ we construct a
valued division algebra $D$  over a local field $F$
with maximal ideal $M_D$ such that
$D^*\big/F^*(1+M_D) \cong
\mathcal D_{q+1}$,
the dihedral group of order $2(q+1)$.
In Section \ref{reducton} we  consider
division algebras with no maximal subgroups. We show that
the assumption of not having maximal subgroups in $D^*$  implies
very strong conditions on $D$ and on its center (Th.~
\ref{mainprop}). Finally in Section \ref{mainsec} we prove that every
quaternion division algebra  has a maximal subgroup, by reducing the
problem to the existence of a maximal subgroup in a quaternion
algebra over a euclidean field; we explicitly construct a
(non-normal) maximal subgroup in this case. By combining
Theorems~\ref{mainprop} and \ref{euclid} in
Sections~\ref{reducton}~and~\ref{mainsec}, we obtain:

\begin{theorem}\label{main}
Let $D$ be a division ring
finite-dimensional over its center $F$, and suppose $D^*$~has no
maximal subgroup. Then,
\begin{enumerate}
\item[(i)] If $\deg(D)$ is even, then $D\cong \oqt\otimes_F E$,
where $E$ is a nontrivial division algebra of odd degree, and $F$
is euclidean $($so $\mychar(F) = 0$$)$ with $F^{*2}$ divisible.
\item[(ii)] If $\deg(D)$ is odd, then $\mychar(F) > 0$,
$\mychar(F) \nmid\deg(D)$, and $F^*$ is divisible. \item[(iii)] In
either case, there is an odd prime $p$ dividing $\deg(D)$; for
each such $p$, we have $[F(\mu_p)\,{:}\,F] \ge 4$ $($so $p\ge
5${}$)$ and the $p$-torsion in $\Br(F)$ is generated by noncyclic algebras
of degree~$p$.
\end{enumerate}
\end{theorem}

This result guarantees the existence of a  maximal subgroup for a
wide range of division algebras. In particular this covers all of
the cases of Cor.~ 2 of \cite{kesh} mentioned above. It also shows
that every  division algebra of degree
 $2^n$ or $3^n$, $n \geq 1$ has a maximal
subgroup.

Throughout this paper, all division rings are finite dimensional
over their centers, hence the use of the terminology division
algebras. By a maximal subgroup of a group we mean a proper
subgroup which is not contained in any other proper subgroup. A
normal maximal subgroup, is a maximal subgroup which is also
normal. 

Recall from the theory of ordered fields (cf.\ \cite{sch}, 
Ch.~3 or \cite{prestel})
that a field $F$ is said
to be {\it formally real} if $F$~admits an ordering, if and only
if $-1$ is not a sum of squares in $F$.  $F$~is said to be {\it
real pythagorean} if every sum of squares is a square in $F$ and
$-1\notin F^{*2}$.  $F$ is said to be {\it euclidean} if $F$ has
an ordering with respect to which every positive element is
a square. Clearly, if $F$ is euclidean then $F$ is real
pythagorean and $F^* = F^{*2} \cup - F^{*2}$, so $F^2 = F^4$.
 Furthermore,
since $\qt$ is split if $a\in F^{*2}$ or $b\in F^{*2}$ and
$\qt \cong \big(\frac{ac^2,bd^2}F\big)$ for any $c,d\in F^*$, the
only quaternion division algebra over a euclidean field $F$ is $\oqt$.

\section{The functor $\CK$ and its relation with   maximal
subgroups}\label{ck1}

Since we are interested in the existence of maximal subgroups,
we first recall what happens for an abelian group.

\begin{lemma}\label{ablemma}
Let $G$ be an abelian group. Then,
\begin{itemize}
\item[(i)]
$G$ has no maximal subgroups if
and only if $G$ is divisible, if and only if $G=G^p$
$($i.e., $G$ is $p$-divisible$)$ for every prime
number $p$.
\item[(ii)]
If $G$ is nontrivial and has bounded exponent, i.e., $G^n = 1$ for some $n$, then
$G$ is not divisible $($so has a maximal subgroup$)$.
\end{itemize}
\end{lemma}

\begin{proof}
 (i) 
The first assertion of (i) is Exercise 1, p.~99 of \cite{fuchs}, and the 
second is (A) on p.~98 of \cite{fuchs}.  Here is the short proof:
If $G$ has a maximal subgroup $M$, then $G/M$ has no nontrivial subgroups,
so $|G/M| = p$ for some prime number $p$.  Then,
$G^p \subseteq M \subsetneqq G$, so $G$ is not $p$-divisible.  Conversely,
if $G \ne G^p$ for some
prime $p$, then $G/G^p$ is a nontrivial vector space over the field~
$\Z/p\Z$; so, $G/G^p$ has a maximal proper subspace, which pulls back
to a maximal subgroup of $G$.  The rest of (i) is clear.

(ii) Suppose $G$ is nontrivial and $G^n = 1$. Then, $G$ has an element of order $p$ for
some prime $p$ dividing $n$.  If $G$ were divisible, then $G$ would have an element
of order $p^m$ for every positive integer $m$. This cannot occur, as $G^n= 1$.
So, $G$ is not divisible.
\end{proof}


There are several ways to attempt to construct (normal) maximal
subgroups for a finite dimensional division algebra $D$ with
center $F$ of degree $n$. Consider the central simple 
matrix algebra
$A=M_t(D)$ where $t$ is a positive integer. The $K$-group
$\CK(A)$ is then defined as
$$
\CK(A) \ =  \ \text{coker}(\K(F) \longrightarrow
\K(A) ).
$$
By Th.~4 (iii), p.~138 of \cite{draxl}, if $A\not \cong M_2(\Bbb
F_2)$ then $K_1(A) \cong K_1(D)$ via the Dieudonn\'e determinant.
Since the Dieudonn\'e determinant is the $t$-power map on the copy of
$F^*$ in $A$, whenever $D$ is noncommutative we have,
\begin{equation}\label{CKformula}
\CK(A)  \
\cong \ D^*/{F^*}^tD'
\end{equation}
 where $D^*$ is the multiplicative group of $D$ and
$D'$ the derived subgroup of $D^*$. Thus $\CK (D)$~is a factor group of
the group $\CK (A)$.
Now, for any $x\in D^*$,  
$x^{-n}\Nrd(x) \in D^{(1)}$ where $\Nrd$ is the reduced norm and
$D^{(1)}=\{d \in D^* \mid \Nrd(d)=1\}$.  Since, further, the reduced Whitehead
group $\SK(D)=D^{(1)}/D'$ is $n$-torsion (by \cite{draxl}, p.~157,
Lemma~2), it follows from \eqref{CKformula}
that $\CK(D)$ is an abelian group of bounded 
exponent $n^2$. (In fact one can show that the bound can be
reduced to $n$, see the proof of Lemma 4, p.~154 in \cite{draxl}
or pp.~579--580 in~\cite{sk1}.) It thus follows
from \eqref{CKformula} that $\CK(M_t(D))$ is
an abelian group of bounded exponent~$tn^2$.
Therefore, if there is a $t$ such that $\CK(M_t(D))$ is nontrivial,
then it  has a normal maximal subgroup by Lemma~\ref{ablemma}(ii);
then,  $D^*$ has a  normal maximal
subgroup. In \cite{hmm} it was conjectured that if $\CK(D)$ is
trivial then $D$ is a quaternion algebra. In \cite{vishne}, in an
attempt to prove this conjecture, it was shown that if $D$ is a
tensor product of cyclic algebras then $\CK(D)$ is trivial if and
only if $D $ is the ordinary quaternion algebra~$\oqt$ over a real
pythagorean field $F$. The non-triviality of the group $\CK$ and other
factor groups of $D^*$ ``close'' to $\CK$ has been studied in
\cite{nk1,vishne,sk1,kesh}.

 There are other ways to deduce that $D^*$ has a
maximal subgroup. For example, if there exists a surjective
homomorphism from $F^*$ to a torsion-free (abelian) group $\Gamma$
such that $\Gamma$~has a maximal subgroup, then one can conclude
that $D^*$ has a  (normal) maximal subgroup. Indeed, let $v\colon
F^* \rightarrow \Gamma$ be a surjective homomorphism. We only need
to consider the case when $\CK(D)$ is trivial, i.e., $D^*=F^*D'$.
Define $w\colon D^* \rightarrow \Gamma$ by $w(d)=v(f)$ where
$d=fd'$, $f \in F^*$ and $d' \in D'$. If $ a \in D' \cap F^* $, then
$1=\Nrd(a)=a^{\deg(D)}$. Since $D'\cap F^*$ is finite, thus a torsion
group, while $\Gamma$ is torsion-free, it follows that $D' \cap F^*
\subseteq \ker(v)$ and that $w$ is a well-defined surjective
homomorphism.  Since $\Gamma$ has a maximal subgroup, it follows
that $D^*$~has a maximal subgroup. From this it follows that if the
center of a division algebra $D$ has a valuation with value group
$\mathbb Z^n$ then $D^*$ has a normal maximal subgroup.  (The case
of this with a discrete rank $1$ valuation is Cor.~8 of
\cite{akbari}.)

The approaches just described  always produce
 normal maximal subgroups of $D^*$
(so subgroups containing~$D'$). However, there exist division algebras with
non-normal maximal subgroups in their multiplicative groups (see
Example~\ref{normal} and Th.~\ref{euclid} below).

The observations above about $\CK$ reduce the question of existence of a
maximal subgroup to consideration of the case when $\CK(M_t(D))$ is
trivial for every
positive integer $t$. In fact we have the following:

\begin{proposition}\label{nonormal} 
Let $D$ be a division algebra with  center $F$. Then the
following are equivalent:
\begin{itemize}
\item[(i)]
$D^*$ has no normal maximal subgroup.
\item[(ii)]
$\CK(M_t(D))=1$ for every positive integer $t$.
 \item[(iii)] $\CK(M_p(D))=1$ for every prime $p$.
\end{itemize}
\end{proposition}
\begin{proof}
(i) $\Rightarrow$ (ii). If $\CK(M_t(D))$ is nontrivial for a
positive integer $t$, then, as pointed out above,  $D^*$ has
a nontrivial  abelian factor group of bounded exponent;
so $D^*$ has normal maximal subgroup by Lemma~\ref{ablemma}(ii).

(ii) $\Rightarrow$ (iii). Clear.

(iii) $\Rightarrow$ (i) (contrapositive). If $D^*$ has a normal maximal subgroup
$N$, then $D^*/N$ is a group with no nontrivial
subgroups; thus, $D^*/N \cong \Bbb Z/p\Z$ for some prime number $p$.
It then follows that ${F^*}^pD' \subseteq N$, so,
$\CK(M_p(D))$ is nontrivial (see \eqref{CKformula}).
\end{proof}

In Section \ref{reducton} below we will see that the
equivalent conditions
on a division algebra $D$ given in Prop.~\ref{nonormal}
 yield very strong constraints on $D$ and on its center.

While $\CK(D)$ is generally quite difficult to compute,
there is a very explicit description of $\CK(\mQ)$ for a
quaternion algebra $\mQ$, which allows us to determine
when $\mQ^*$ has a normal maximal subgroup. Recall that if
$\mQ$ is a quaternion algebra over a field $F$ with $\mychar(F)
\ne 2$, then for some $a,b\in F^*$, $\mQ \cong \qt$, where
$\qt$ denotes the quaternion algebra over $F$ with \mbox{$F$-base}
$\{\, 1, \qi, \qj, \qk\,\}$ satisfying $\qi^2 = a$, $\qj^2 = b$,
and $\qk = \qi\qj = -\qj \qi$.  For any ${x = r+ s\qi + t\qj+ u\qk
\in \qt}$ (with $r,s,t,u\in F$), the reduced norm of $x$ is given by
\begin{equation}\label{quatnorm}
\Nrd(x) \ = \ r^2 -as^2 - bt^2 +abu^2\,.
\end{equation}
Note that if $b\in F^{*2}$, then the quaternion algebra is split.
If $\mychar(F) = 2$, then every quaternion algebra over $F$ has the
form $\left[\frac{c,b}F\right)$ for $c\in F$ , $b\in F^*$;
this is the $F$-algebra with $F$-base $\{\, 1, \qi, \qj, \qk\,\}$
satisfying $\qi^2 - \qi = c$, $\qj^2 = b$, and $\qk =\qi\qj = \qj \qi + \qj$.
Here again, if $b \in F^{*2}$ the quaternion algebra is split.

\begin{lemma}\label{quat}
Let
$\mQ$ be a quaternion division algebra
over a field $F$.
Then,
$ \mQ^* \big/ \mQ' \cong \Nrd(\mQ^*)$
and, for every $t$,
\begin{equation}\label{CKquat}
\CK(M_t(\mQ))   \ \cong \ \mQ^*\big/F^{*t}\mQ'
\ \cong \ \Nrd(\mQ^*) \big/ F^{*2t}\,.
\end{equation}
\end{lemma}

\begin{proof}  The first isomorphism is given in \eqref{CKformula}
above. For the second, recall that 
$\SK(\mQ) = \mQ^{(1)}\big/ \mQ'$, where $\mQ^{(1)}
= \ker(\Nrd)$.  Since $\mQ$ is a quaternion algebra, it is
known that $\SK(\mQ)$~is trivial, see Th.~1, p.~161 in \cite{draxl}.
(In fact, every element of $\mQ^{(1)}$ is a commutator.) Consequently,
$\mQ^* \big/ \mQ' \cong \Nrd(\mQ^*)$.  Since $\Nrd(F^{*t}) = F^{*2t}$,
it follows that ${\mQ^*\big/F^{*t}\mQ' \cong \Nrd(\mQ^*) \big/ F^{*2t}}$.
\end{proof}


\begin{proposition}\label{euclidean}
Let $\mQ$ be a quaternion division algebra with center $F$. Then
the following are equivalent:
\begin{itemize}
\item[(i)]
$\mQ^*$ has no  subgroup of index $2$.
\item[(ii)]
The group $\CK(M_2(\mQ))$ is trivial.
\item[(iii)]
$F$ is a euclidean field and $\mQ\cong \oqt$.
\end{itemize}
\end{proposition}

\begin{proof}
(i) $\Rightarrow$ (ii) (contrapositive).  As noted above
(and explicitly clear from Lemma~\ref{quat}), $\CK(M_2(\mQ))$
is a $4$-torsion abelian group.  If $\CK(M_2(\mQ))$ is nontrivial,
then by Lemma~\ref{ablemma}
it has a maximal subgroup  $N$, which
is necessarily normal and of prime index, say $p$.  Since
$\CK(M_2(\mQ))\big/N$ is $4$-torsion, we must have $p = 2$.  Thus,
the inverse image of $N$ in $\mQ^*$ has index $2$ in $\mQ^*$.

(ii) $\Rightarrow$ (iii).  Suppose $\CK(M_2(\mQ))$ is trivial.  Then,
$\Nrd(Q^*) = F^{*4}$ by Lemma~\ref{quat}.  Since $F^{*2} = \Nrd(F^*)
\subseteq \Nrd(\mQ^*) = F^{*4}$, we have $F^{*2} = F^{*4}$. If $\mychar(F)
= 2$, then ${F = (F^2)^{1/2} = (F^4)^{1/2}} = F^2$,
i.e., $F$ is perfect.
But, since $\mQ \cong \left[ \frac{c,b}F\right)$ and $b\in F^* = F^{*2}$,
$\mQ$ is split.  This cannot occur since $\mQ$ is assumed to be a division
algebra. Hence, $\mychar(F) \ne 2$, so  $\mQ \cong \qt$ for some
$a,b\in F^*$. Since $\Nrd(\mQ^*) = F^{*2}$, formula~\eqref{quatnorm}
shows that $-a, -b \in F^{*2}$ and every sum of squares in $F$ is a
square.  Also, $-1\notin F^{*2}$, since otherwise $b = (-1)(-b) \in
F^{*2}$  and $\mQ$~would be split.  Hence, $F$ is real pythagorean.
Because $F^{*4} = F^{*2}$, for every $c\in F^*$ there is $d\in F^*$
with $c^2 = d^4$; then $c = \pm d^2$.  So, $F^* = F^{*2} \cup - F^{*2}$
(a disjoint union).  This shows that every positive element of $F$ with
respect to any ordering must be a square.  So, $F$ is euclidean.
Therefore, as noted above, $\mQ \cong \oqt$.

(iii) $\Rightarrow$ (ii). Suppose $F$ is euclidean, so $\mQ\cong \oqt$.
Then, by the reduced norm formula~\eqref{quatnorm}, $\Nrd(\mQ^*) = F^{*2}
= F^{*4}$, as $F$ is euclidean.  Hence, $\CK(M_2(\mQ))$ is trivial, by
Lemma~\ref{quat}.

(ii) $\Rightarrow$ (i) (contrapositive).  Suppose $\mQ^*$ has a subgroup
$H$ of index $2$.  Then, $H$ is normal in $\mQ^*$ with $\mQ^*\big /H
\cong \Z/2\Z$. Hence, $\mQ'\subset H$ and $F^{*2} \subseteq H$.  Therefore,
$F^{*2}\mQ' \subseteq H\subsetneqq \mQ^*$, so  $\CK(M_2(\mQ))$ is
nontrivial.
\end{proof}

In Section \ref{mainsec} we will show that the quaternion division
algebra over a euclidean field always has (non-normal) maximal
subgroups and thus conclude that every quaternion division algebra
over any field has a maximal subgroup. For the moment, we will
describe exactly when a quaternion division algebra has a normal
maximal subgroup.  This will enable us to give examples of
quaternion division algebra over certain euclidean fields which have
normal maximal subgroups (necessarily of odd prime index, by
Prop.~\ref{euclidean}).

\begin{proposition}\label{euclidnonormal}
Let $\mQ$ be a quaternion division algebra with center $F$.
For any odd prime~$p$, if $F^* \ne F^{*p}$, then
$\mQ^*$ has a normal maximal subgroup of index $p$ or $2$.
Hence,
$\mQ^*$ has no normal maximal subgroup if and only if  $F$ is
euclidean and $F^* = {F^*}^p$ for every odd prime $p$.
\end{proposition}

\begin{proof}
Suppose $\mQ$ has no subgroup of index $2$. Prop.~\ref{euclidean}
shows that this occurs iff $F$ is eucidean.  Also by
Prop.~\ref{euclidean},
$\CK(M_2(\mQ))$
is trivial, so $\mQ^* = F^{*2}\mQ'$, by \eqref{CKformula}. Hence,
\begin{equation}\label{iso}
\mQ^* \big/ \mQ' \ = \ F^{*2}\mQ'\big / \mQ'  \ \cong \  F^{*2}\big/
(F^{*2} \cap \mQ') \,.
\end{equation}
If $a\in F^{*2} \cap \mQ'$, then $a>0$ and $a^2 = \Nrd(a) = 1$; so, $F^{*2} \cap \mQ'
= \{1\}$. 
As noted previously, a normal maximal subgroup of $\mQ^*$ has prime
index and contains $\mQ'$. Thus, if $p$ is any odd prime, $\mQ^*$~has a
normal maximal subgroup iff $F^{*2}$ has a maximal subgroup of index $p$
iff $F^{*2} \ne F^{*2p}$ (see Lemma~\ref{ablemma}), iff $F^* \ne F^{*p}$.
\end{proof}

In the next two examples we will work with valued division algebras $D$.
A valuation on~$D$ is an epimorphism $v
\colon D^* \rightarrow \Gamma_D$, where
$\Gamma_D$ is a totally ordered abelian group, such that if
$v(a)\geq 0$ then $v(a+1) \geq 0$ (this is equivalent to the
traditional definition). Let $V_D={\{a\in D^* \mid v(a) \geq 0\}
\cup \{0\}}$, the valuation ring of $D$, and
let  ${M_D=\{a\in D^* \mid v(a) > 0\}
\cup \{0\}}$, the unique maximal left and right ideal of $V_D$. Thus
$\overline D= V_D/M_D$ is a division ring called the residue
division ring, and $U_D=V_D^*=V_D \backslash M_D$ is the group
of valuation units.
The restriction of $v$ to $F = Z(D)$ induces a valuation on $F$ and gives
the corresponding structures $V_F, M_F,U_F,\overline F$ and
$\Gamma_F$. (For a survey of valued division algebras see
\cite{wadval}.)

Prop.~\ref{euclidnonormal} shows that the multiplicative group of
Hamilton's quaternion division algebra~$(\frac{-1,-1}{\Bbb R})$
has no normal maximal subgroup.
The next example shows that the quaternion
division algebra $\mQ = \oqt$ over a euclidean field $F$
can have normal
maximal subgroups (of odd index, by Prop.~\ref{euclidean}),
i.e., by Prop.~\ref{nonormal}, there is a positive integer $t>2$ 
such that $\CK(M_t(\mQ))$ is
nontrivial (recall that here ${\CK(M_2(\mQ))=1}$).

\begin{example}\label{maxeuc} 
Let $K$ be any field with an ordering $<$, and let $R$ be a real
closure of $K$ with respect to $<$; let $<$ denote also the unique
ordering on $R$.  Let $F$ be the euclidean hull of $K$ in $R$.  That
is, $F = \bigcup\limits_{i = 0}^\infty L_i$, where $L_0 = K$ and
for each $i\ge 0$, ${L_{i+1} = L_i\big(\{\sqrt c \,|\ c\in L_i, \
c>0\,\}\big) \subseteq R}$.  By construction, the ordering on $R$
restricts to  an ordering on $F$ in which each positive element
of $F$ is a square; so, $F$ is euclidean.  Take any odd prime $p$.
Let $E$ be any quadratic extension field of $K$.  The composition of
maps $K^*\big/K^{*p} \to E^*\big/E^{*p}
\stackrel{N}\rightarrow  K^*\big/K^{*p}$ (where $N$~is induced by
the norm $N_{E/F}$) is the squaring map, which is an isomorphism
 as $p$ is odd.   Hence, the map $K^*\big/K^{*p} \to E^*\big/E^{*p}$
is injective.  Thus, the map $K^*\big/K^{*p} \to F^*\big/F^{*p}$
is  an injection, as $F$~is the direct limit of iterated quadratic
extensions of $K$.  Therefore, whenever $K^{*p} \ne K^*$
the quaternion division algebra $\oqt$ over our euclidean field $F$
has a normal maximal subgroup of index $p$.  For example, when
$K = \Q$, the field $F$ is the field of constructible numbers, in the
sense of compass and straightedge constructions, and $\oqt$
has a normal maximal subgroup of index $p$ for every  odd prime
$p$.  For another example, let $K = \Bbb R((x))$, the Laurent series
field in one variable over the real numbers $\Bbb R$.  Then,
with respect to the ordering on $K$ with $x>0$, the euclidean hull is
$F = K(\{\,\root {2^n}\of x\,| \ n = 1,2, \ldots\,\})$; this $F$
has a Henselian (but not complete) valuation induced by the
$x$-adic valuation on~$F$, with value group $\Gamma_F$ isomorphic
to the additive group of the ring $\Z[1/2]$ and 
residue field~$\ol F \cong\R$.
For every odd prime $p$, we have $F^*\big/F^{*p} \cong \Gamma_F\big/
p\Gamma_F \cong \Z/p\Z$.  The valuation on $F$ extends uniquely
to a valuation $v$ on $\mQ = \oqt$ with   $\ol\mQ \cong \left(\frac
{-1,-1}{\R}\right)$ and $\Gamma_\mQ = \Gamma_F$. For each odd
prime $p$, $\{\, a\in \mQ^*\,| \  v(a) \in p\Gamma_{\mQ} \,\}$ is
the unique normal subgroup of $\mQ$ of index $p$, and these are all the
normal maximal subgroups of $\mQ^*$.

In each of these examples, $\CK(\mQ)$ and $\CK(M_2(\mQ))$ are
trivial by Prop.~\ref{euclidean}, but $\CK(M_3(\mQ))$ is nontrivial
 by the proof
of Prop.~\ref{nonormal},
as $\mQ^*$ has a normal maximal subgroup of index $3$.

\end{example}

We next give examples of
division algebras with non-normal maximal subgroups of finite
index.

\begin{example} \label{normal} Let $q$ be
any prime power. We construct a division algebra $D$ with center
a local field $F$ such that
\begin{displaymath}
{D^*}\big/ {F^*(1+M_D})\ \cong  \ \mathcal D_{q+1}.
\end{displaymath}
Here $\mathcal D_{q+1}$ is the dihedral group with $2(q+1)$
elements,
and
$M_D$ is the maximal ideal of the valuation ring of $D$.
Note that for any $n
>2$,  the dihedral group $\mathcal D_n$ has nonnormal maximal subgroups of
index $p$ for each odd prime $p$ dividing $n$ (and these are the
only nonnormal maximal subgroups). It thus follows that
for each odd prime $p$ dividing $q+1$
there is a
maximal subgroup $H$ in $D^*$ of index $p$ such that $F^*(1+M_D)
\subseteq H$ but $H$ is not normal in~$D^*$.

For this example, we first observe an
exact sequence, \eqref{localsequence} below, relating a homomorphic 
image of
$D^*$ to value group and residue data.  The sequence is exact for any 
valued division
algebra $D$ finite dimensional over its center $F$.
Note that since ${U_D \cap F^*(1+M_D)  = U_F(1+M_D)}$,
there is  a short exact sequence,
\begin{equation}\label{seq}
1 \
 \ \longrightarrow  \ U_D\big /U_F(1+M_D)
 \ \longrightarrow  \ D^*\big/ F^*(1+M_D)
 \ \longrightarrow  \ D^*\big/ F^*U_D
 \ \longrightarrow  \ 1
\end{equation}
Now, the
reduction epimorphism $U_D \to \ol D^*$ has kernel $1+M_D$,
and likewise ${\ol F^* \cong U_F\big/ (1+M_F)}$.
Hence, $U_D\big/ U_F(1+M_D) \cong \ol D^*\big/ \ol F^*$.
Also, the epimorphism $D^* \to \Gamma_D\big/ \Gamma_F$ induced
by the valuation has kernel $F^*U_D$.  By plugging this information into
\eqref{seq} we obtain the short exact sequence.
\begin{equation}\label{localsequence}
 1 \
\longrightarrow  \ {{\overline D }^*}\big/{ {\overline F}^*} \
\longrightarrow  \ {{D^*}} \big/ {{F^*}(1+M_D)} \
\stackrel{v}{\longrightarrow}  \ {\Gamma_D}\big/ {\Gamma_F}
 \ \longrightarrow \  1.
\end{equation}
Thus, $\big|D^*\big/F^*(1+M_D)\big| <\infty$ iff 
$\big| {\overline D }^*\big/ {\overline F}^* \big| <\infty$
and $\big|\Gamma_D/\Gamma_F\big| < \infty$. Note that if 
$\ol D \ne \ol F$, 
then  $\big| {\overline D }^*\big/ {\overline F}^* \big| <\infty$
iff $|F|< \infty$. 

Now, take a field $F$ with a discrete  rank $1$ valuation $v$, i.e.,
$\Gamma_F = \Z$. Let
$L$ be a cyclic Galois field extension of $F$
of degree $n$, and let $\Gal(L/F) = \langle \sigma \rangle$.
Suppose $L$ is unramified over $F$, i.e.,
$v$~has a unique extension from $F$ to $L$ with
$\ol L$ separable of degree $n$ over $\ol F$.
Take any $\pi \in F^*$ with $v(\pi) = 1$, and let $D$ be the cyclic algebra
$D = (L/F, \sigma,\pi)$.  So, $D = \bigoplus\limits _{i = 0}^{n-1}Lx^i$,
where $xcx^{-1} = \sigma(c)$ for all $c\in L$, and $x^n = \pi$.
It is known, see Cor.~2.9 in \cite{jwncp},
and easy to verify, that $v$ extends to a valuation
on $D$ given by $v\big(\sum\limits_{i=0}^{n-1}c_ix^i\big) =
\min\limits_{0\le i \le n} \big(v(c_i) + i/n\big)$.  Hence, $D$ is a
division ring, with $\ol D = \ol L$ and $\Gamma_D = \frac 1n \Z$.
Note that $v(x) = 1/n$, so that the image of $v(x)$
generates the cyclic group $\Gamma_D\big/\Gamma_F
\cong \Z/n\Z$.  But also, $x^n = \pi\in F^*$, so the image
$\widetilde x = xF^*(1+M_D)$ of $x$ in $D^*\big/F^*(1+M_D)$
has order dividing $n$.  Therefore, there is a well-defined
homomorphism $\Gamma_D\big/\Gamma_F \to D^*\big/F^*(1+M_D)$
sending $1/n + \Gamma_F$ to $\widetilde x$; this is a splitting map
for the short exact sequence \eqref{localsequence}. Hence, the middle
group in \eqref{localsequence} is a semidirect product,
\begin{equation}\label{semidirect}
D^*\big/F^*(1+M_D) \ \cong \ \ol L^*\big/ \ol F^*  \ltimes  \Z/n\Z\,,
\end{equation}
where the conjugation action of the distinguished generator of
$\Z/n\Z$ on $\ol L^*\big/\ol F^*$ in the semidirect product is
induced by the automorphism of $\ol L$ induced by $\sigma$
on $L$.

To be more specific, let $q= \ell^m$ for any prime $\ell$ and any positive
integer $m$, and let $F$ be the unramified extension of degree $m$
of the $\ell$-adic field $\Q_\ell$.  With respect to the (complete, discrete
rank $1$) valuation $v$ on $F$ extending the $\ell$-adic valuation on
$\Q_\ell$, we have $\ol F \cong \F_q$, the finite field with $q$
elements.  Let $L$ be the unramified extension of $F$ of degree
$n$.  Then, with respect to the unique extension of $v$ to
$L$, we have $\ol  L \cong \F_{q^n}$, and $L$ is cyclic Galois
over~$F$ as the valuation is Henselian and $\ol L$ is cyclic Galois
over $\ol F$.  Let $\sigma$ be the Frobenius automorphism of $L$,
which is the generator of $\Gal (L/F)$ which induces the
$q$-th power map on $\ol L$.  Since $\ol L^*$ is a cyclic group,
the isomorphism of \eqref{semidirect} becomes
\begin{equation}\label{localcase}
D^*\big/F^*(1+M_D) \ \cong \ \big[\Z\big/\big((q^n-1)/(q-1)\big)\Z\big]
\,\ltimes  \Z/n\Z\,,
\end{equation}
where the distinguished generator of $\Z/n\Z$ acts on
$\Z\big/\big((q^n-1)/(q-1)\big)\Z$ by
multiplication by $q$.
If we specialize to $n = 2$, then $D$ is the unique quaternion
division algebra over $F$, and
multiplication by $q$  on  $\Z/(q+1)\Z$
coincides with the inverse map, so the right group in \eqref{localcase}
is the dihedral group $\mathcal D_{q+1}$. 
\end{example}

\begin{remark}
Let $D$ be a {\it strongly tame} valued division algebra over a 
Henselian field $F$, i.e., 
${\mychar (\overline F) \nmid \deg(D)}$.
Then,
$1+M_D=(1+M_F)[D^*,1+M_D]$ (see the proof of Th.~3.1 in
\cite{sk1}). It follows that $F^*(1+M_D)=F^*[D^*,1+M_D]$.
Also if $\mychar(\overline F) \not = 2$, then by Th.~21 in \cite{Riehm},
$[D^*,1+M_D]=D''$ where $D''=[D',D']$. Putting these together, if in
the above example $n= 2$ and $q$ is not a $2$-power, then
$$D^*/F^*D''
\cong \mathcal D_{q+1}.$$

\end{remark}

\bigskip


\section{Maximal subgroups of $D^*$---reduction to the quaternion
case}\label{reducton}

Let $F$ be a field. For any $m \in  \mathbb N$, let $\mu_m(F)$
denote the group of all $m$-th roots of unity in $F$. Also $\mu_m
\subseteq F$ means that $F$ contains a primitive $m$-th root of
unity i.e., $\mu_m(F)$ has order $m$.

For a prime number $p$, ${}_p\text{Br}(F)$  denotes the $p$-torsion
subgroup of the Brauer group $\Br(F)$, and $\Br(F)(p)$
denotes the $p$-primary component of $\Br(F)$.

Throughout this section, $D$ is a non-commutative
division algebra finite dimensional over its center $F$.
Recall from Prop.~\ref{nonormal} that if $D$ has no (normal)
maximal subgroup (of prime finite index) then $\CK(M_k(D))$ is trivial
for every $k \in \Bbb N$. The goal of this section is to prove the
following theorem:

\begin{theorem}\label{mainprop}
Let $D$ be a division algebra  of degree $n$, with center $F$,
such that the group $\CK(M_k(D))$ is trivial for every positive
integer $k$. Then,
\begin{enumerate}
\item[(i)] if $n$ is odd, then $\mychar(F) > 0$ and $\mychar(F)
\nmid n$ and for each prime number $q$, $F^*={F^*}^q;$

\item[(ii)] if $n$ is even, then $n=2m$ with $m$ odd,
$\mychar(F)=0$, $F$ is euclidean, $F^*={F^*}^q$ for each odd
prime, and $\Br(F)(2)={}_2\text{\rm{Br}}(F)=\{F,\oqt\};$ 

\item[(iii)] in either case, for each odd prime $p$ dividing $n$,
$\mu_p \nsubseteq F$, and ${}_p\text{\rm{Br}}(F)$ contains $($and is generated
by$)$ noncyclic algebras of degree $p$, and $\DEG{F(\mu_p)}F\ge 4$.
\end{enumerate}
\end{theorem}

The proof will be given below, after some preliminary steps.

\begin{lemma}\label{dec}
Let $D$ be a division algebra with center $F$, where $D$ has degree
$n=p_1^{r_1} \ldots p_l^{r_l}$ with the $p_i$ distinct primes. If
$\CK(M_k(D))$ is trivial for every positive integer $k$ then
${{F^*}^{p_i^{r_i}}={F^*}^{p_i^{r_i+1}}}$, $1\leq i \leq l$ and
$F^*={F^*}^q$ for every prime $q$ other than the $p_i$.
\end{lemma}
\begin{proof}
Since $\CK(M_k(D))\cong D^*/{F^*}^kD'$, if $\CK(M_k(D))$ is
trivial then, $D^*={F^*}^kD'$ for any $k \in \Bbb N$. Applying
$\Nrd$ to this equation, we get:
$$
{F^*}^n \ \subseteq \ \Nrd(D^*)\ =\ \Nrd({F^*}^kD')\ =\ {F^*}^{nk}
\ \subseteq\  {F^*}^n.
$$
Thus for every $k \in \mathbb N$,
\begin{equation}\label{tri}
{F^*}^n\ =\ {F^*}^{nk}.
\end{equation}
The Lemma then follows from (\ref{tri}) and Lemma \ref{lem2}
below.
\end{proof}

\begin{lemma} \label{lem2} Let $A$ be an abelian group, written additively. Let
$n=p_1^{r_1} \ldots p_l^{r_l}$ with the $p_i$ distinct primes and
$r_i \geq 1$. The following are equivalent:
\begin{enumerate}
\item[(i)] $nA=nkA$ for every $k \in \mathbb N$.

\item[(ii)] $p_i^{r_i}A=p_i^{r_i+1}A$ for each $i$, and $A=qA$ for
every prime $q$ different from the $p_i$.
\end{enumerate}
\end{lemma}
\begin{proof}
(ii) $\Rightarrow$ (i) is clear. (It suffices to check (i) for $k$
a prime number.)

(i) $\Rightarrow$ (ii). Note that for any $s,t \in \mathbb N$ with
$\gcd(s,t)=1$, we have
\begin{equation} \label{abelianst}
\textstyle {A}\big/{stA} \ = \ \big({sA}\big/{stA}\big) \
\bigoplus \ \big({tA}\big/{st A}\big)\,.
\end{equation} For, as $\text{gcd}(s,t)=1$, $A=sA+tA$, and
$\big(sA/stA\big) \cap \big(tA/st A\big) =(0)$, since $sA/stA$ is $t$-torsion and
$tA/stA$ is $s$-torsion. Now, take any prime $q$ different from
the $p_i$. Since $\text{gcd}(q,n)=1$, (\ref{abelianst}) and (i)
yield
$$
\textstyle{A}\big/{nqA} \ = \ \big( {nA}\big/{nqA}\big)\
\bigoplus \ \big({qA}\big/{nqA}\big)\ =\
 (0) \
\bigoplus\ \big({qA}\big/{nqA}\big)\,.
$$
So, $A=qA$. Now, for $1 \leq i \leq
l$, write $n=p_i^{r_i}u$, so $np_i=p_i^{r_i+1}u$ with
$\text{gcd}(p_i^{r_1+1},u)=1$. Then (\ref{abelianst}) shows
$$
\textstyle{A}\big/{np_iA} \ =\ \big({p_i^{r_i+1}A}\big/{np_iA}\big)\
\bigoplus\ \big({uA}\big/{np_iA}\big)\, .
$$
Multiplying this by $p_i^{r_i}$:
\begin{align*}
\textstyle
{p_i^{r_i}A}\big/{np_iA}\ &=\
\big(\big[{p_i^{r_i}(p_i^{r_i+1}A)+np_iA}\big]\big/{np_iA}\big) \
\textstyle\bigoplus\
\big({nA}\big/{np_iA}\big) \ \\
&\subseteq\  \big({p_i^{r_i+1}A}\big/{np_iA}\big)\ \textstyle\bigoplus \ (0)
\ \subseteq\  {p_i^{r_i}A}\big/{np_iA}.
\end{align*}
So,
$p_i^{r_i}A=p_i^{r_i+1}A$.
\end{proof}

\begin{lemma}\label{pr=pr1} If $p$ is a prime number and $r \in \mathbb N$, then
${F^*}^{p^r}={F^*}^{p^{r+1}}$ if and only if
$F^*=\mu_{p^r}(F){F^*}^p$. 
\end{lemma}

\begin{proof}
Suppose ${F^*}^{p^r}={F^*}^{p^{r+1}}$. Take any $a \in F^*$. There
is a $b \in F^*$ with $a^{p^r}=b^{p^{r+1}}$. Let $\omega=ab^{-p}$. Then
$\omega \in \mu_{p^r}(F)$, and $a=\omega b^p \in \mu_{p^r}(F){F^*}^p$.
So, $F^*=\mu_{p^r}(F){F^*}^p$. The converse is clear.
\end{proof}

\begin{proposition}\label{propbrauer} 
Let $F$ be a field with ${F^*}^{p^r}={F^*}^{p^{r+1}}$ for some
prime $p$ and some $r \in \mathbb N$, and suppose ${}_p{\Br(F)}  
\neq (0)$. Then,
\begin{enumerate}
\item[(i)] if $p$ is odd, then $\mychar(F) \neq p$, $F^*={F^*}^p$,
$\mu_p \nsubseteq F$, and ${}_p{\Br(F)}$ is generated by noncyclic
algebras of degree $p;$

\item[(ii)] if $p=2$, then $\mychar(F)=0$, $F$ is euclidean and
$\Br(F)(2)={}_2{\Br(F)}=\{F,\oqt\}$.

\end{enumerate}
\end{proposition}

\begin{proof}
Let $\omega$ be a generator of the cyclic group $\mu_{p^r}(F)$. By
Lemma \ref{pr=pr1}, $F^*=\langle \omega \rangle{F^*}^p$.

(i) Assume $p$ is odd. If $\mychar(F)=p$, then by Albert's theorem
(see \cite{albert}, p.~109, Th.~30 or \cite{jacobson}, p.~173, Th.~4.5.7),
$\Br(F)(p)$ is generated by cyclic algebras of degree a power of~
$p$. But when $\mychar(F)=p$ we have $\omega=1$, so
$F^*={F^*}^p$, i.e., $F$ is perfect, so 
every generator of $\Br(F)(p)$ is split.
This contradicts the assumption that ${}_p\text{Br}(F)  \neq (0)$. Hence,
$\mychar(F) \neq p$. If $\mu_p \subseteq F$, then the
Merkurjev-Suslin theorem  (see \cite{suslin} or
\cite{gille}, Ch.~8) says that ${}_p \text{Br}(F)$ is
generated by $p$-symbol algebras. Since $F^*/{F^*}^p=\langle
\omega{F^*}^p\rangle$, we would then have ${}_p \text{Br}(F)$ is a
cyclic group generated by the $p$-symbol algebra
$(\omega,\omega;F)_p$. But $(\omega,\omega;F)_p\cong
(\omega,-1;F)_p$, so that $(\omega,\omega;F)_p$ is both
$p$-torsion and $2$-torsion in ${}_p \text{Br}(F)$, so it must be split.
This cannot occur since ${}_p \text{Br}(F)\neq 0$. Hence $\mu_p
\nsubseteq F$. Therefore $\omega=1$ and $F^*={F^*}^p$. By a
theorem of Merkurjev (see \cite{merk}, Th.~2), since
$\mychar(F)\neq p$, ${}_p\text{Br}(F)$ is generated by algebras of
degree~$p$. Since $F^*={F^*}^p$, these generators cannot be cyclic
algebras. (Of course, the existence of noncyclic division algebras
of prime degree is a major open question).

(ii) Assume now that $p=2$. As in case (i), if $\mychar(F)=2$,
then $F$ is perfect, so that ${}_2\text{Br}(F)=(0)$ by Albert's theorem,
contrary to hypothesis. So, $\mychar(F)\neq 2$. By Merkurjev's
Theorem (see, e.g., \cite{kersten}, Kap.~V for a proof), ${}_2\text{Br}(F)$ 
is generated by quaternion algebras. Since
$F^*=\langle \omega \rangle {F^*}^2$, ${}_2\text{Br}(F)$ must be a cyclic
group generated by the quaternion algebra $(\frac{\omega,\omega}{F})$. But
$(\frac{\omega,\omega}{F})=(\frac{\omega,-1}{F})$. If $\mu_4 \subseteq F$,
then
$-1 \in {F^*}^2$, so that $(\frac{\omega,-1}{F})$ is split; then
${}_2\text{Br}(F)=0$, a contradiction. Hence $\mu_4 \nsubseteq F$,
forcing $\omega=-1$, and $-1 \not \in {F^*}^2$. Since $F^*=\langle \omega
\rangle {F^*}^2$, we have $F^*={F^*}^2\cup -{F^*}^2$ (a disjoint
union). Also, ${}_2\text{Br}(F)=\{[F],[H]\}$ where $H=\oqt$ which is
nonsplit. It follows that $\mychar(F)=0$. For, if
$\mychar(F)=q\not = 0$, then $H$ is split, since already over the
prime field $\mathbb F_q$, $(\frac{-1,-1}{\mathbb F_q})$ is split.

Let $\qi$ and $\qj$ be the standard generators of $H$. Take any 
$a,b \in F^*$. Then $a^2+b^2=\Nrd(a+b\qi) \not = 0$ as $H$ is a
division ring. Hence, there is $c \in F^*$ with $a^2+b^2=\pm c^2$.
If $a^2+b^2=-c^2$, then $0=a^2+b^2+c^2=\Nrd(a+b\qi+c\qj)$, which
cannot occur, as $H$ is a division ring. Therefore, $a^2+b^2=c^2$.
Hence, $F$ is pythagorean. Since $-1 \not \in {F^*}^2$, $-1$ is
therefore not a sum of squares. Therefore, $F$ is formally real.
Since $F^*={F^*}^2\cup -{F^*}^2$, $F$ is in fact euclidean. Now,
let $L=F(\sqrt{-1})$. By Hilbert's Th.~90, we have the exact
sequence
\begin{equation}\label{hilbert}
{F^*}\big/{{F^*}^2} \ \longrightarrow \ {L^*}\big/{{L^*}^2}
\ \longrightarrow  \ {F^*}\big/{{F^*}^2},
\end{equation}
where the left map is induced by the inclusion $F^*\hookrightarrow
L^*$, and the right map is induced by the norm $N_{L/F}$. For $a,b
\in F^*$, we have $N_{L/F}(a+b\sqrt{-1})=a^2+b^2 \in {F^*}^2$.
Thus, the right map in (\ref{hilbert}) is the $0$-map. The left
map in (\ref{hilbert}) is also $0$-map, since $-1\in {L^*}^2$.
Hence, $L^*={L^*}^2$. Therefore, Merkurjev's Theorem shows that
${}_2\text{Br}(L)=0$, so $\Br(L)(2)=0$. Hence, $\Br(F)(2) \subseteq
\Br(L/F)\ (\,=\ker(\Br(F)\rightarrow \Br(L))\,).$ But, as $\DEG LF=2$,
$\Br(L/F)\subseteq {}_2\text{Br}(F)$. Thus, $\Br(F)(2)={}_2\text{Br}(F)$.
\end{proof} 

\begin{lemma}\label{zerochar}
Suppose $\mychar(F)=0$ and $F^*={F^*}^q$ for each prime $q$. Then
$\mu_q \subseteq F$ for each prime $q$.
\end{lemma}

\begin{proof} This is Lemma 3 of \cite{may}.  We include the
short proof for the convenience of the reader.
The proof is by induction on $q$. Of course $\mu_2=\{\pm 1\}
\subseteq F$. Now assume $q>2$ and $\mu_\ell\subseteq F$ for all
primes $\ell<q$. We have $F(\mu_q)$ is an abelian Galois extension of
$F$ with $[F(\mu_q)\co F]\,\big|\,(q-1)$. If $F(\mu_q)\not = F$, then there
is a prime $p \divides (q-1)$ and a sub-extension $F \subseteq K
\subseteq F(\mu_q)$ with $[K\co  F]=p$, so $K$ is cyclic Galois over
$F$. Since $\mu_p \subseteq  F$ by induction, $K$ is a $p$-Kummer
extension of $F$. This cannot occur, as $F^*={F^*}^p$. Hence,
$F(\mu_q)=F$, as desired.
\end{proof}

We can now prove Theorem~\ref{mainprop}.

\noindent {\it Proof of Theorem~\ref{mainprop}.} Since
$\CK(M_k(D))$ is trivial for every $k \in \Bbb N$, Lemma \ref{dec}
shows that $F^*={F^*}^q$ for each prime $q$ with $q \nmid n$. Let
$p$ be an odd prime with $p \divides n$. Then ${}_p\text{Br}(F)  \not =
0$ since  it contains some nonsplit tensor power of $D$. So,
Lemma~\ref{dec} and Prop.~\ref{propbrauer}(i) show that
$F^*={F^*}^p$, $\mu_p \nsubseteq F$, and ${}_p\text{Br}(F)$ is generated
by noncyclic algebras of degree $p$. This last condition implies
$[F(\mu_p)\co F]\ge 4$, by the Corollary to Th.~1 in \cite{merk}.
Suppose $n$ is odd. Then $F^*={F^*}^q$ for every prime $q$. Since
$\mu_p \nsubseteq F$ for any prime $p$ with $p \divides n$, Lemma
\ref{zerochar} shows that $\mychar(F)\not=0$. Also Lemma~\ref{dec}
and Prop.~ \ref{propbrauer}(i) show that $\mychar(F) \nmid n$.
This completes the proof of (i) and (iii)
of Th.~\ref{mainprop}. For (ii) assume now
that $n$ is even. Lemma~\ref{dec} and Prop.~\ref{propbrauer}(ii)
show that $\mychar(F)=0$, $F$ is euclidean, and
${}_2\text{Br}(F)=\Br(F)(2)=\{F,\oqt\}$. Since the $2$-primary component
of $D$ therefore must be $\oqt$, $n/2$ must be odd.
\qed

\bigskip

\begin{remark}
The result of Lemma \ref{zerochar} is definitely not true in prime
characteristic, since cyclic Galois extensions of degree
$\mychar(F)$ are Artin-Schreier extensions, not Kummer extensions.
For example, let $p$ be a prime number, and let $\mathbb F_p$ be
the finite field with $p$ elements. In an algebraic closure
of $\mathbb F_p$, let $L_i$ be the field
with $[L_i\co \mathbb F_p]=i$ for all $i \in \mathbb N$, and let
$F=\bigcup_{p\nmid i}L_i$;
so the supernatural number $[F\co\mathbb F_p]$ is the product of $q^{\infty}$
for all primes $q \ne p$.  
 Then, $F^*={F^*}^q$ for every prime $q$,
but for those primes  $q$ with $p\mid [\mathbb F_p(\mu_q)\co \mathbb
F_p]$, we have $\mu_q \nsubseteq F$. (E.g., if~$p=3$, then $\mu_7
\nsubseteq F).$)
\end{remark}

\section{Maximal subgroups of the  multiplicative group of a quaternion
algebra}\label{mainsec}

In this section we shall prove that the multiplicative group of a
quaternion division algebra contains maximal subgroups. We will
see that the most difficult case is that of quaternion algebras
over euclidean fields.  As shown by Prop.~\ref{euclidnonormal},
such division algebras may not have any normal maximal subgroups.


\begin{theorem}\label{mainthm}
Let $\mQ$ be a quaternion division algebra with center $F$. Then
the multiplicative group of $\mQ$  has a 
 maximal subgroup.
\end{theorem}

\begin{proof}

If $\mQ$ has no normal maximal subgroup, then by Prop.~\ref{euclidean}, $\mQ=
\oqt$, where $F$ is a euclidean field.
 We will
show in Th.~\ref{euclid} below that such a  $\mQ$ nonetheless has a 
nonnormal maximal
subgroup.  That will complete the proof of this theorem.
\end{proof} 

The rest of this section is devoted to showing that the
quaternion division algebra $\oqt$ over a euclidean field
$F$ contains a
(non-normal) maximal subgroup. This will be done by a refinement of the
argument given in \cite{mahdavi}, attributed to C.~Ohn, showing that
for $F= \rr$, $\left(\frac{-1,-1}\rr\right)$ has a maximal
subgroup.  Significant added complexity arises here because we
need to take into account the possible existence of infinitesimals
with respect to the ordering on~$F$.  A different proof that
$\left(\frac{-1,-1}\rr\right)$ has  maximal subgroups is given in
\cite{akbari4}.
\medskip

Let $F$ be a euclidean field. Then $F$ has a valuation ring $V$
which is determined by the ordering:
$$
V=\{b \in F \mid |b| \leq n \text{ for some } n \in \Bbb N\}\, ,
$$
whose maximal ideal is
$$
M=\{b\in F \mid |b| \leq 1/n \text{ for every } n \in \Bbb N\}
$$
(see, e.g., \cite{sch} p.~135).
Note that $F \backslash V$ is the set of elements ``infinitely
large'' relative to the rational numbers $\Bbb Q \subseteq F$.
Also, $M$ is the set of
elements of $F$ ``infinitesimal'' relative to~$\Bbb Q$.

We will need some geometric properties for inner product spaces
 over the euclidean field~$F$, which are familiar when the field is $\rr$.

For any $n \in \Bbb N$, let $F^n=\{(a_1,a_2,\ldots,a_n) \mid a_i
\in F\}$. For $\alpha=(a_1,a_2,\ldots,a_n)$ and
$\beta=(b_1,b_2,\ldots,b_n)$ in $F^n$, we have the dot product:
$\alpha \cdot \beta=a_1b_1+a_2b_2+\ldots+a_nb_n \in F$. The norm
$\|\alpha\|= \sqrt{a_1^2+a_2^2+\ldots+a_n^2}=\sqrt{\alpha \cdot
\alpha} \in F$ (as $F$ is euclidean). Note that the following
basic tools carry over to this setting: The Cauchy-Schwarz
inequality: $| \alpha \cdot \beta | \leq \| \alpha \| \| \beta
\|$, and the triangle inequality: $\| \alpha+\beta \|\leq \|
\alpha \|+\| \beta \|$. We write $\alpha \bot \beta $ if $\alpha
\cdot \beta =0$.

Now let
$$
\Orth(F) \ = \ \{A \in M_n(F) \mid A^t A=I\}
$$
and
$$
\SO(F) \ = \ \{ A \in \Orth(F) \mid \det(A)=1\}\,.
$$
So, for any $A\in \Orth(F)$ and any $\alpha, \beta\in F^n$,
we have
$(A\alpha\cdot A\beta) = (\alpha\cdot \beta)$ and $\|A\alpha\|
= \|\alpha\|$.
 Clearly,
$$
\textstyle{\SO[2](F)\ =\ \big \{ \left(%
\begin{smallmatrix}
  c & -s \\
  s & c \\
\end{smallmatrix}%
\right) \mid c,s \in F, \ c^2+s^2=1\big \}}
$$ 
is an abelian group,
whose elements can be thought of as ``rotations''. Also,
$$
\Orth[2](F)\, \backslash\, 
\SO[2](F) \ =\ \big \{ \left(%
\begin{smallmatrix}
  c & s \\
  s & -c \\
\end{smallmatrix}%
\right) \mid c,s \in F, \,c^2+s^2=1\big \}.$$ Each $A \in \Orth[2](F)
\backslash \SO[2](F)$ has eigenvectors 1,-1 with orthogonal
eigenspaces. So, $A$ is then a reflection.

Note that as $F$ is euclidean, $\SO[2](F)$ is $2$-divisible. For,
if $A=
\left(\begin{smallmatrix}  c & -s \\
  s & c \\
\end{smallmatrix}\right)
\in \SO[2](F)$, with $c^2+s^2=1$, then $|c| \leq 1$, so
$c+1 \geq 0$.
Let $B= \left(%
\begin{smallmatrix}
  a & -b \\
  b & a \\
\end{smallmatrix}%
\right)$, where, if $c=-1$ (so $s=0$), $a=0$ and $b=\pm 1$, while
if $c \not = -1$, $a=\pm \sqrt{(c+1)/2}$ and $b=s/2a$. Then
$a^2+b^2=1$ (and $a,b\in F$ as $c+1 \geq 0$ and $F$ is euclidean),
so $B \in \SO[2](F)$, and $B^2=A$. Basically we are just invoking
the half-angle formula from trigonometry. 

Now, let $A \in
\SO[3](F)$. Observe  that as $A^tA = I$ in $M_3(F)$, we have
\begin{align*}
\det(A-I) \ &= \ \det((A-I)^t) \ = \ \det(A^t - A^tA) \ = \ \det(A^t) 
\det(I -A) \\  &= \ 1\cdot (-1)^3\det(A-I) \ = \ -\det(A-I)\, .
\end{align*}
Since $\mychar(F) \ne 2$, this shows that $\det(A-I) = 0$, 
proving that $1$ is an eigenvalue of $A$.  Let $v$ in $F^3$ be a
$1$-eigenvector of $A$, and enlarge $\{\,v\,\}$ to an orthogonal base 
$\mathcal B = \{\, v, v_2, v_3\,\}$ of~$F^3$. The matrix of $A$ as a 
linear transformation on $F^3$ relative to the base $\mathcal B$ is
 $\left(%
\begin{smallmatrix}
  1 & 0 \\
  0 & D \\
\end{smallmatrix}%
\right)$ where $D \in \Orth[2](F)$, and $\det(D)=\det(A)/1=1$; so $D
\in \SO[2](F)$, i.e., $D$ is a ``rotation.''  Thus we can think of
$A$ as a rotation about the axis determined by the 1-eigenvector
$v$. Because $D$~is the square of a matrix in $\SO[2](F)$,
$A$ is the square of a matrix in $\SO[3](F)$. Thus $\SO[3](F)$ is
$2$-divisible (though non-abelian).

Let $\mH=\oqt$ be the ordinary quaternion division algebra over
$F$ with its standard base $\{1,\qi,\qj,\qk\}$
and standard involution 
given by $\overline{a+b\qi+c\qj+d\qk} = a-b\qi-c\qj-d\qk$.
We identify $\mH$
with $F^4$ via
$
\ a+b\qi+c\qj+d\qk \ \leftrightarrow  \ (a,b,c,d)\,.
$
Then, for $x
\in \mH$, we have $\| x \|=\sqrt{\Nrd(x)}$, where for
$x=a+b\qi+c\qj+d\qk,$
$$
\Nrd(x) \ = \ x \overline x \ = \  a^2+b^2+c^2+d^2\,.
$$
Note also that the reduced trace of $x$ is $\Trd(x) =x+\overline x = 2a$.

Let
$$
S(\mH) \ = \  \{ x \in \mH \mid \|x \|=1 \}
$$
be the unit
sphere in $\mH$. Let $P
=\{b\qi+c\qj+d\qk  \mid b,c,d \in F\}$, the ``purely imaginary part'' of~
$\mH$. Note that
\begin{equation}\label{Pdesc}
P \ =\ \{ \alpha \in \mH \mid \Trd(\alpha)=0\} \
= \ \{ \alpha \in\mH \mid \alpha^2 \in F,
\alpha \not \in F \}\cup \{0\}\,.
\end{equation}

Let
$$
S(P) \ = \ \{ \alpha \in P \mid \|\alpha \|=1\}\, ,
$$
the unit
sphere in $P$. The geometry in $P$ is nicely tied to the
multiplication: A straightforward calculation shows that for
$\alpha=a_1\qi+a_2\qj+a_3\qk$ and $\beta=b_1\qi+b_2\qj+b_3\qk \in
P$, we have
\begin{equation}\label{alphabeta}
\alpha\beta \ = \  -(\alpha\cdot\beta) \, + \,\alpha \times \beta,
\end{equation}
where the cross product $\alpha \times \beta$
is the formal determinant
$$
\alpha \times \beta \ = \
\left|%
\begin{array}{ccc}
  \qi & \qj & \qk \\
  a_1 & a_2 & a_3 \\
  b_1 & b_2 & b_3 \\
\end{array}%
\right|\, \in P\, .
$$
 Since $\beta \times \alpha =- \alpha \times
\beta$, formula \eqref{alphabeta} shows that $\alpha \cdot \beta
= -\frac{1}{2}(\alpha
\beta+\beta\alpha)$. Thus, $\alpha \bot \beta$ if and only if
$\alpha$ and $\beta$ anticommute.

Now, $\mH ^*$ acts on $\mH$ by conjugation: For $x \in \mH^*, y
\in \mH$, set
$$
x * y\, =\, x y x^{-1}\,.
$$
Note that since conjugation
preserves the reduced norm, it also preserves the norm, i.e., $\|x
* y\|=\|y\|$, and hence it also preserves the dot product, i.e.,
$$
(x* y)\cdot(x * z) \ = \ y\cdot z
$$
(as $2(y\cdot z)=\|y+z\|^2-\|y\|^2-\|z\|^2).$ Note
that for $x \in \mH^*$ and $ \alpha \in P$,
by \eqref{Pdesc} above we have
$x * \alpha \in P$,
since $\Trd(x * \alpha)=\Trd(\alpha)=0$ (or, $x* \alpha \not \in
F$ as $\alpha \not \in F$ (assuming $\alpha \not =0$), but
$(x*\alpha)^2=x * (\alpha^2)=\alpha^2 \in F$). Thus, the
conjugation action of
$\mH^*$ on $\mH$ restricts to an action of $\mH^*$ on $P$, which
is norm- and dot product-preserving. So, $\mH^*$ also acts on the
unit sphere $S(P)$. There is a very nice geometric description of
this action, as follows:

Take any $x \in \mH^*$. Since conjugation by $x$ coincides with
conjugation by $\frac{1}{\|x\|}x$, we may assume that $\|x \|=1$.
Then we can write $x=c+sp$, for some $c,s \in F, p \in P$ with
$\|p\|=1$, so 
$c^2+s^2=1$ as $\|x\|=1$. If $s=0$, then $x \in F$, so $x *
\alpha=\alpha$ for all $\alpha \in P$. So, assume $s\not =0$.
Then, $s$~and $p$ are unique up to factor of $-1$.

For $\{p\}^{\bot}=\{y \in \mH \mid y\cdot p=0\},$ we have
$\dim_F(\{p\}^{\bot} \cap P)=2$. So, there is $q \in S(P)$ with $q
\bot p$. Set $r=pq$. We have $p^2=-\|p\|=-1,q^2=-1$, and
$r=pq=-qp$; hence, $r^2=-1, qr=-rq=p$, and $r p=-p r=q$. From this,
it is clear that there is an \mbox{$F$-automorphism}  of $\mH$ given by
$\qi \mapsto p$ and $\qj \mapsto q$. In particular, $\{p,q,r\}$ is
an orthogonal base of $P$. Since $\|x \|=1$ and $x=c+sp$, we have
$x^{-1}=\overline x=c-sp$. Thus, for any $\alpha= a p+bq+dr$ where
$a,b,d \in F$, we have
\begin{equation}\label{star}
\begin{aligned}
x * \alpha\  &= \ (c+sp)\,(ap+bq+dr)\,(c-sp)
\\
& =\ ap\,+\,[(c^2-s^2)b-2csd]q\,+\,[2csb+(c^2-s^2)d]r.
\end{aligned}
\end{equation}
That is, the matrix of the $F$-linear transformation $\alpha
\mapsto x*\alpha$ of $P$ relative to the orthogonal base
$\{p,q,r\}$ is
$$
\left(%
\begin{matrix}
  1 & 0 & 0 \\
  0 & {\ c^2-s^2 \ } & -2cs \\
  0 & 2cs & {\ c^2-s^2\ } \\
\end{matrix}%
\right)\, .
$$
Heuristically, think of $c=\cos(\theta)$ and $s=\sin(\theta)$ for
some imagined angle $\theta$, so that $\left(%
\begin{smallmatrix}
  c & -s \\
  s & c \\
\end{smallmatrix}%
\right)$ is the matrix for rotation by $\theta$. Then, $x*$ is
rotation by an angle $2\theta$ about the $p$-axis.

Let's imagine the $2$-sphere $S(P)$ oriented so that $\qi$ is at
the north pole and
$$
E\,=\,F\text{-span}\{\qj,\qk\} \cap S(P)
$$
 is the
equator. Take any $\beta=c_1\qj+s_1\qk \in E$, so $c_1^2+s_1^2=1$,
and choose 
${B=\left(%
\begin{smallmatrix}
  c_0 & -s_0 \\
  s_0 & c_0 \\
\end{smallmatrix}%
\right) \in \SO[2](F)}$ such that $B^2=\left(%
\begin{smallmatrix}
  c_1 & -s_1 \\
  s_1 & c_1 \\
\end{smallmatrix}%
\right)$ (recall that $\SO[2](F)$ is 2-divisible). Then, for
$y=c_0+s_0 \qi$, formula~(\ref{star}) shows that $y*\qj=\beta$.
Thus, the $\mH^*$-orbit of $\qj$ contains all of $E$. Similarly, for
any $\gamma \in S(P)$,  take a two dimensional subspace $W$
of $P$ containing $\qj$ and $\gamma$, and choose $p \in S(P)$ with
$p \bot W$. Then we can take $q=\qj$ and $r = p\qj$; $\{q, r\}$
is an orthonormal base of~$W$, so we have  $\gamma=c_1q+s_1r$
with $c_1^2+s_1^2=1$.
From the $2$-divisibility of $\SO[2](F)$, as above, there exist
 $c_0,s_0\in F$  with $
c_0^2 + s_0^2 = 1$, $c_0^2 - s_0^2 = c_1$ and $2c_0s_0 =s_1$;
if we   set $x=c_0+s_0p$, then formula~\eqref{star} shows that
$x*\qj=\gamma$. Thus, $\mH^*$ acts transitively on $S(P)$.

\begin{theorem} \label{euclid} Let $F$ be a euclidean field, and
let $\mH=\oqt$. Then
$\mH^*$ contains a maximal  subgroup.
\end{theorem}
\begin{proof} Recall that $M$ is the set of $\Bbb
Q$-infinitesimal elements of $F$. Let
$$
\Delta \ = \ \{ \alpha \in S(P) \mid \|\alpha -\qi  \|\in M\},
$$
the set
of elements of $S(P)$ ``infinitesimally near'' $\qi$.

Let $C=\{a+b\qi \mid a,b \in F\} \cong F(\sqrt{-1})$ which is the
centralizer of $\qi$ in $\mH$. Let
$$
G_0 \ = \ C^* \ = \ \{a+b\qi \mid a \not
= 0 \text{ or } b \not =0 \} \  \subseteq  \ \mH^*,
$$
which is the
stabilizer of $\qi$ in $\mH^*$. For each $a \in F$ with $|a| \leq
1$, let
$$
L_a \ =  \  \{a\qi+b\qj+d\qk \in P \mid b^2+d^2=1-a^2\},
$$
the ``$a$-latitude'' on $S(P)$. We saw above that $G_0$ acts
transitively on $E=L_0$, and an analogous argument shows that
$G_0$ acts transitively on each $L_a$.
 Since  $\qj *(a\qi+ b\qj + d\qk) =
-a\qi+ b\qj -  d\qk$, we have
$\qj * \qi=- \qi$ and $\qj*L_a=L_{-a}$ for each
$a$-latitude.

 Let
$$
G\ = \ \{x \in \mH^* \mid x*\qi \in \Delta \cup -\Delta \}\, .
$$
Then,
$G_0 \subseteq G$ and $\qj \in G$. Note that for $x \in \mH^*$, if $x*\qi \in
\Delta$, then $x*\Delta \subseteq \Delta$. For, if $\alpha \in
\Delta$, then
\begin{align*}
\|x*\alpha-\qi \| \  &\leq  \ \|x*\alpha
-x*\qi\|\,+\,\|x*\qi-\qi \| \ = \  \|x*(\alpha-\qi) \|\,+\,\|x*\qi-\qi
\|\\
&= \ \|\alpha-\qi\|\,+\,\| x*\qi-\qi\| \ \in \  M +M\, ,
\end{align*}
 so $x* \alpha \in
\Delta$. Likewise, if $x*\qi \in -\Delta$ then $x*\Delta \subseteq
-\Delta$ and $x*(-\Delta) \subseteq \Delta$. Therefore,
$G$ is closed under multiplication.  Furthermore, for
$\epsilon = \pm 1$, we have $\|\overline x*\qi-\epsilon \qi\|
=\|-\big(\overline{x*\qi-\epsilon \qi}\big)\| = \|x*\qi-\epsilon \qi\|$.
Hence, if $x\in G$, then $\overline x \in G$.  Because $x^{-1} =
\frac1{\|x\|}\overline x$ and $F^*\subseteq G$, it follows that $G$ is closed
under inverses;  hence, $G$ is a subgroup of $\mH^*$.
 Since $\mH^*$ acts
transitively on $S(P)$ but $G*\qi \subseteq \Delta \cup -\Delta
\subsetneqq S(P)$, $G$ must be a proper subgroup of $\mH^*$.

\medskip
\noindent {\bf Claim 1}: $G$ is a maximal  subgroup of
$\mH^*$.

\noindent {\it Proof of Claim 1}. Take any $y \in \mH^* \backslash
G$, and let $K= \langle y,G\rangle$. We show that $K=\mH^*$ by
proving that $K*\qi=S(P)$. For then, for any $h \in \mH^*$, there is
$z \in K$ with $h*\qi=z*\qi$. Then $z^{-1}h*\qi=\qi$, so that
$z^{-1}h \in G_0 \subseteq K$; hence $h=z(z^{-1}h) \in K$. Let
$$
y\ = \ r+t\qi+u\qj+v\qk \ = \ (r+t\qi)+(u+v\qi)\qj\, ,
$$
 with $r,t,u,v \in F$.
Replacing $y$ by $y\qj$ if necessary (without changing $K$, as
$\qj \in G$), we can assume $r+t\qi \not =0$. Then (as $r+t\qi \in
G_0 \subseteq G$), we can replace $y$ by $(r+t\qi)^{-1}y$, so we
can assume $t=0$. Furthermore, as $F^* \subseteq G$, we can
replace $y$ by $\frac{1}{\|y\|}y$ without changing $K$; so we can
assume that $\|y\|=1$. Thus, $y=c_0+s_0p$, where $c_0,s_0 \in F$ with
$c_0^2+s_0^2=1,p \in P, \|p\|=1$ and $p\bot\qi$. Without loss of generality,
we may assume that $p=\qj$. (For, if $p \not =\qj$, we can work with
the orthonormal base 
$\{\qi,p,\qi p\}$ of $P$ instead of $\{\qi,\qj,\qk\}$,
and the same argument as below clearly goes through.) Thus,
$y=c_0+s_0\qj$ where $c_0,s_0 \in F$ with $c_0^2+s_0^2=1$. Formula
(\ref{star}) then yields, for any $a,e,d \in F$,
\begin{equation}\label{star2}
\begin{aligned}
y*(a\qi+e\qj+d\qk)\ &= \ (c_0+s_0\qj)\,(a\qi+e\qj+d\qk)\,(c_0-s_0\qj)
\\
&= \ [(c_0^2-s_0^2)a+2c_0s_0d]\qi\,+\,e\qj\,+\,[-2c_0s_0a+(c_0^2-s_0^2)d]\qk
\\
&= \ (ca+sd)\qi\,+\,e\qj\,+\,(-sa+cd)\qk,
\end{aligned}
\end{equation}
where
$$
c\,=\,c_0^2-s_0^2 \text{ \ and  \ } s\,=\,2c_0s_0
$$
 (so, $c^2+s^2=1$). In
particular $y*\qi=c\qi-s\qk$.
We have $c,s \in V$ (the valuation ring) since $c^2+s^2=1$
shows $|c| \leq 1$ and $|s| \leq 1$.
Note further that
$$
\|y*\qi-\qi\|^2 \ = \ (c-1)^2\,+\,s^2 \ = \ (c-1)^2\,+\,(1-c^2)
\ =\ 2(1-c),
$$
and
likewise $\|y*\qi+\qi\|^2=2(1+c)$. Since $y *\qi \not \in \Delta$
and $y*\qi \not \in -\Delta$ by hypothesis, we must have $1+c
\not \in M$, $1-c \not \in M$; hence $s^2=(1-c)(1+c) \not \in
M$, so $s \not \in M$.

By replacing $y$ by $y\qj$ if necessary (which interchanges $|c_0|$
and $|s_0|$), we may assume $c\geq 0$. Also, by replacing $y$ by
$y^{-1}$ if necessary (which replaces $s_0$ by $-s_0$ without changing
$c_0$), we may assume $s \geq 0$.

\medskip
\noindent {\bf Claim 2}: Finitely many applications of elements of
$K$ map $\qi$ to any point on any latitude $L_b$, for any $b$ with
$0 \leq b \leq 1$.

Since $\qj*L_b=L_{-b}$, it follows from Claim 2 that
$K*\qi=\bigcup_{-1\leq b \leq 1}L_b=S(P)$, which, as we
showed above, proves Claim 1.

\noindent {\it Proof of Claim 2}. Recall that for $|a| \leq 1$,
$$
L_a \ = \  \{a\qi\pm \sqrt{1-a^2-d^2}~\qj +d\qk \mid |d| \leq \sqrt{1-a^2}\}.
$$
Thus, formula (\ref{star2}) shows that for $0\leq a \leq 1$,
\begin{equation}\label{star3}
\text{$(y*L_a)\cap L_{b}\ne \varnothing$  \  \ for every $b\in F$ with  \ \ }
ca -s \sqrt{1-a^2}  \ \leq  \ b  \ \leq \  ca+s\sqrt{1-a^2}.
\end{equation}

If we set $a = c$, so $\sqrt{1-a^2} = s$, condition~\eqref{star3}
says that $L_{c}$ meets $L_b$ for all $b$ with
$c^2 - s^2 \le b\le c^2 + s^2 = 1$;
in particular, this holds for $c \le b\le 1$, since $c^2-s^2
\le c^2 \le c \le 1$ (as $0\le c\le 1$).
Now, $y*\qi \in L_{c}$,
and $G_0$ acts transitively on $L_{c}$; so, $(G_0y)*\qi = L_c$.
For any~$b$ with $c \leq
b \leq 1$, we have just seen that $y*$ maps some point on $L_{c}$
to a point on $L_{b}$. Also $G_0$~acts transitively on $L_b$. So,
for any such $b$, $L_b \subseteq (G_0yG_0y) * \qi \subseteq
K*\qi$. Thus, $K*\qi$~contains all latitudes above~$L_{c}$.

To handle  the latitudes below $L_{c}$, we will need:
\begin{equation}\label{formii}
\text{ If } 0 \leq a \leq c, \text{ then  \ }
ca-s\sqrt{1-a^2}  \ \leq \  a-s^2
 \ \leq \  a  \ \leq \  ca+s\sqrt{1-a^2}.
\end{equation}
To see this, note that since $0\leq a\leq c \leq 1$, we
have $\sqrt{1-a^2} \geq \sqrt{1-c^2}=s$. Thus,
$s^2\le  s\sqrt{1-a^2}$. Since 
$ca \leq a$, this yields
the first inequality in \eqref{formii}.
The second inequality in \eqref{formii} is clear. The third inequality in
\eqref{formii} is equivalent to $2a^2 \leq 1+c$, which holds as $2a^2 \leq
2a$ (as~$0 \leq a \leq 1$) and $2a \leq 1+c$ (as $ a\leq c$ and
$a\leq 1$).

The inequalities in ~\eqref{formii} combined with (\ref{star3})
show that for all $a\in F$ with $0 \leq a \leq c$,
\begin{equation}\label{formiii}
 \text{for all $b\in F$ with $a-s^2 \leq b
\leq a$,}  \ \  \ \   (y*L_a) \cap L_b\,\ne \,\varnothing,
\text{  \ so \  } L_b\, \subseteq \, (G_0y)*L_a .
\end{equation}
 Thus (taking $a = c$ in \eqref{formiii}), for $b$ with for $c-s^2 \leq b
\leq c$, we have
$$
L_b \ \subseteq(G_0y)*L_{c} \ = \  (G_0y)^2*\qi \ \subseteq K*\qi\, .
$$
This proves Claim 2  if $c-s^2\le 0$, so we may assume
$c>s^2$.
For an integer $k\ge1$, with $c -ks^2\ge 0$,
suppose $L_b  \subseteq  K*\qi$ for $c-ks^2 \le b\le c$.
Then (taking $a = c-ks^2$ in~\eqref{formiii}), for all
$b$ with $c-(k+1)s^2 \le b\le
c - ks^2$, we have
 $$
L_b  \ \subseteq  \ (G_0y)*L_{c-ks^2} \ \subseteq \  (G_0y)K*\qi
 \ = \  K*\qi.
$$
Hence, $L_b \subseteq K*\qi$ for $c-(k+1)s^2\le b\le c$.
It follows by induction that
for all positive integers $n \le c/s^2$,
$L_b  \subseteq K*\qi$ for  all $b$ with
$c-(n+1)s^2 \le b\le c$.

Because $s\notin M$, $s$ is a unit of the valuation ring
$V$; so $c/s^2\in V$.  Hence, by the definition of $V$, there is
a positive integer $m$ with $c/s^2 <m$.  Let $n+1$ be the smallest
such $m$.  Then, $n \le c/s^2 < n+1$, and $n\ge 1$ as $c/s^2 >1$.
For this $n$, since $c-(n+1)s^2 \le 0$, we proved in the previous
paragraph that for all $b$ with
$0\le b\le c$, we have $L_b\subseteq K*\qi$.  We proved this inclusion earlier
for $b$ with
$c\le b\le 1$.  This proves
Claim 2,  completing the proof of Claim 1 and Th.~\ref{euclid}.
\end{proof}

 Th.~\ref{mainprop} and Th.~\ref{euclid}
combine to yield Th.~\ref{main} stated in the Introduction. 
This theorem shows that to produce an example of a $D^*$ with no maximal subgroup, one
would have to find a field with a noncyclic division algebra of prime
degree.  The existence of such noncyclic division algebras is one of the
oldest and most challenging open questions in the theory of division algebras.

{\bf Acknowledgements.} The first named author would like to
acknowledge the support of Queen's University PR grant and EPSRC
EP/D03695X/1.
Part of the work for the paper was done while he was visiting the
second named author at the University of California at San Diego
in the Summers 2005 and 2006. He would like to thank him for his
care and attention.


\begin{thebibliography}{AEKG}

\bibitem [AEKG]{akbari4} S. Akbari, R. Ebrahimian, H. Momenaee Kermani, A. Salehi
Golsefidy, \paper{Maximal subgroups of $\GL[n](D)$}, J. Algebra,
{\bf 259} (2003), 201--225.

\smallskip

\bibitem[AMM]{akbarimah} S. Akbari, M. Mahdavi-Hezavehi, M. G. Mahmudi, \paper{Maximal subgroups
of $\GL[1](D)$}, J. Algebra, {\bf 217} (1999), 422--433.

\smallskip


\bibitem[AM]{akbari} S. Akbari, M. Mahdavi-Hezavehi, \paper{On the existence of normal
maximal subgroups in division rings,} J. Pure Appl. Algebra, {\bf
171} (2002), 123--131.

\smallskip

\bibitem[A]{albert} A. Albert, {\it Structure of Algebras}, Amer. Math.
Soc. Colloquium Publ., Vol. 24, Providence, 1961.

\smallskip



\bibitem[Am]{amt} A. Amitsur, \paper{Finite subgroups of division
rings}, Trans. Amer. Math. Soc., {\bf 80}, (1955), 361--386.

\smallskip


\bibitem [D]{draxl} P. Draxl,  {\it Skew Fields},
London Mathematical Society Lecture Note Series {\bf 81},
Cambridge University Press, Cambridge, 1983.

\smallskip


\bibitem[E]{ebrahimi} R. Ebrahimian, \paper{Nilpotent maximal subgroups of
$\GL[n](D)$}, J. Algebra, {\bf 280} (2004), 244--248.

\smallskip

\bibitem[F]{fuchs} L. Fuchs, {\it Infinite Abelian Groups}, Vol. 1, 
Academic Press,  New York, 1970.

\smallskip

\bibitem[GS]{gille} P. Gille, T. Szamuely, {\it Central Simple Algebras
and Galois Cohomology}, Cambridge Univ. Press, Cambridge, 2006.

\smallskip

\bibitem[H]{sk1} R. Hazrat, \paper{$\SK$-like functors for division
algebras}, J. Algebra, {\bf 239} (2001), 573--588.

\smallskip

\bibitem[HMM]{hmm} R. Hazrat, M. Mahdavi-Hezavehi, B. Mirzaii,
\paper{Reduced $K$-theory and the group
$G(D)=D^{*}/F^{*}D'$}, 
pp.~403--409 in
Algebraic $K$-Theory and its Applications, ed. H. Bass,
World Sci. Publishing, River Edge, NJ, 1999.

\smallskip

\bibitem[HV]{vishne} R. Hazrat, U. Vishne,
\paper{Triviality of the functor $coker(K_1(F)\rightarrow K_1(D))$
for division algebras}, Comm.  Algebra, {\bf 33} (2005),
1427--1435.

\smallskip 

\bibitem[HW]{nk1} R. Hazrat, A. Wadsworth, \paper{Nontriviality of certain
quotients of $K_1$ groups of division algebras}, J. Algebra, {\bf
312} (2007), 354--361.

\smallskip


\bibitem[JW$_1$]{jwncp} B. Jacob, A. Wadsworth, \paper{A new construction
of noncrossed product algebras}, Trans. Amer. Math. Soc.,
{\bf 293} (1986), 693--721.

\smallskip

\bibitem[JW$_2$]{jacobwads} B. Jacob, A. Wadsworth, \paper{Division algebras
over Henselian fields}, J. Algebra, {\bf 128} (1990), 126--179.

\smallskip

\bibitem[J]{jacobson} N. Jacobson, {\it Finite-Dimensional Division Algebras},
Springer-Verlag, Berlin, 1996.

\smallskip

\bibitem[K]{kersten} I. Kersten, {\it Brauergruppen von K\"orpern}, Vieweg,
Braunschweig, Germany, 1990.

\smallskip

\bibitem[KM]{kesh} T. Keshavarzipour, M. Mahdavi-Hezavehi,
\paper{ On the non-triviality of $G(D)$ and the existence of
maximal subgroups of $\GL[1](D)$}, J. Algebra, {\bf 285} (2005),
213--221.

\smallskip

\bibitem[L]{lam} T.-Y. Lam, {\it A First Course in Noncommutative
Rings}, Springer-Verlag, New York, 1991.

\smallskip


\bibitem[M]{mahdavi} M. Mahdavi-Hezavehi, \paper{Free subgroups in
maximal subgroups
of $\GL[1](D)$}, J. Algebra, {\bf 241} (2001), 720--730.

\smallskip

\bibitem[Ma]{may} W. May, \paper{Multiplicative groups of fields},
Proc. London Math. Soc. (3), {\bf 24} (1972), 295--306.

\smallskip

\bibitem[Me]{merk} A. S. Merkurjev, \paper{Brauer groups of fields},
Comm. Algebra, {\bf 11} (1993), 2611--2624.

\smallskip






\bibitem[P]{prestel} A. Prestel, {\it Lectures on Formally Real fields},
Lecture Notes in Math., No. 1093, Springer-Verlag, Berlin, 1984.

\smallskip



\bibitem[RSS]{rapsegec} A. Rapinchuk, Y. Segev, G. Seitz,
\paper{Finite quotients of the multiplicative group of a
 finite dimensional division algebra are solvable}, J. Amer.
 Math. Soc., {\bf 15} (2002), 929--978.

\smallskip

 \bibitem[R]{Riehm} C. Riehm, \paper{The norm 1 group of a $p$-adic
 division algebra}, Amer. J. Math., {\bf 92} (1970),
 499--523.

\smallskip


\bibitem[Sch]{sch} W. Scharlau,
{\it Quadratic and Hermitian Forms}, Springer-Verlag, Berlin, 1985.

\smallskip

\bibitem[Sco]{scott} W. Scott, \paper{On the multiplicative group of a
division ring}, Proc. Amer. Math Soc., {\bf 8} (1957),
303--305.

\smallskip



\bibitem[S]{suslin} A. A. Suslin, \paper{Algebraic $K$-theory and
the norm residue homomorphism}, Itogi Nauki i Tekhniki,
Akad. Nauk SSSR, Vsesoyuz. Inst. Nauchn. i Tekhn. Inform.,
{\bf 25} (1984), 112--207 (Russian); English trans.:
J. Soviet Math., {\bf 30} (1985), 2556--2611.

\smallskip

\bibitem[St]{stuth} C. Stuth, \paper{A generalization of the Cartan-Brauer-Hua
theorem}, Proc. Amer. Math Soc., {\bf 15} (1964), 211--217.


\smallskip

\bibitem[W]{wadval} A. Wadsworth, \paper{Valuation theory on finite dimensional
division algebras}, pp. 385--449 in Valuation Theory and its
Applications, Vol. I, eds. F.-V. Kuhlmann et. al., Fields Inst.
Commun. {\bf 32}, Amer. Math. Soc., Providence, RI, 2002.

\end{thebibliography}
\end{document}